\documentclass[a4paper,reqno,12pt]{amsart}
\usepackage{amsmath,amsthm}
\usepackage{amssymb,mathrsfs,dsfont}
\usepackage{paralist}
\usepackage{xspace}
\usepackage[a4paper,
centering, left=92.27pt,right=92.27pt,
textheight=649.5pt,
headsep=24pt,
headheight=38pt
]{geometry} 

\usepackage{times}

\ifx\pdfoutput\@undefined\usepackage[usenames,dvips]{color}
\else\usepackage[usenames,dvipsnames]{color}
\IfFileExists{pdfcolmk.sty}{\usepackage{pdfcolmk}}{} 
\fi

\usepackage{graphicx}

\usepackage{ifpdf} 
\ifpdf
\DeclareGraphicsExtensions{.pdf,{}}
\else
\DeclareGraphicsExtensions{.eps,.ps,jpg,jpeg,{}}
\fi
\usepackage[
unicode]{hyperref}
\hypersetup{ 
    bookmarksnumbered={true},%
    bookmarksopen={true},
    pdftitle =Escape Rate of Diffusion Processes,
    pdfauthor ={Ouyang, S.-X.},
    pdfsubject={escape rate, volume growth, Bakry-Emery curvature and comparsion theorem},
    pdfkeywords ={escape rate, diffusion process, Dirichlet space, Riemannian manifold,volume growth},
    pdfdisplaydoctitle=true, 
    colorlinks = true,
    linkcolor = black,
    anchorcolor = black,
    citecolor = black,
    filecolor = black,
    urlcolor = black
} 

\newtheorem{theorem}{Theorem}[section]
\newtheorem{prop}[theorem]{Proposition}
\newtheorem{lemma}[theorem]{Lemma}
\newtheorem{cora}[theorem]{Corollary}

\theoremstyle{definition}

\newtheorem{assumption}[theorem]{Assumption}
\newtheorem{example}[theorem]{Example}
\newtheorem{remark}[theorem]{Remark}

\def\<{\langle} 
\def\>{\rangle}

\renewcommand{\H}{\ensuremath{\mathds{H}}\xspace}
\DeclareMathOperator{\e}{e}

\DeclareMathOperator\cut{cut}

\def\loc{\mathrm{loc}}

\def\M{\mathbb M}

\def\m{\mu}

\def\E{\mathds E}

\def\F{\mathscr F}
\def\G{\mathscr G}

\def\P{\mathds P}

\def\SS{\mathbb S}

\def\C{\mathscr C} 
\def\PP{\mathscr P}

\def\C{C}

\def\EC{\mathcal E}

\def\B{\mathscr B}

\def\R{\mathds R}

\DeclareFontFamily{U}{bbold}{}
\DeclareFontShape{U}{bbold}{m}{n}{<-5>bbold5<6>bbold6<7>bbold7<8>bbold8
<9>bbold9<9-11>bbold10<11-14>bbold12<14->bbold12}{}
\DeclareSymbolFont{bbold}{U}{bbold}{m}{n}
\DeclareSymbolFontAlphabet{\mathbbold}{bbold}
\DeclareMathSymbol{\Eins}{\mathord}{bbold}{`1}

\numberwithin{equation}{section}

\begin{document}

\title[Escape Rate of Diffusion Process]{Volume growth, Comparison theorem and Escape Rate of Diffusion Process}
\author{Shunxiang Ouyang}
\address{Department of Mathematics, Bielefeld University, 33615, 
Bielefeld, Germany}
\email{souyang@math.uni-bielefeld.de}

\begin{abstract}
We study the escape rate of diffusion process with two approaches. We first give an upper rate function for the diffusion process associated with a symmetric, strongly local regular Dirichlet form. The upper rate function is in terms of the volume growth of the underlying state space. The method is due to Hsu and Qin [Ann. Probab., 38(4), 2010] where an upper rate function was given for Brownian motion on Riemannian manifold.  In the second part, we prove a comparison theorem and give an upper rate function for diffusion process on Riemannian manifold in terms of the upper rate function for the solution process of a one dimensional stochastic differential equation.  
\end{abstract}

\thanks{Supported by SFB 701 of the German Research Council}
\date{\today }

\maketitle


\section{Introduction} 
Let $X_t$, $t\geq 0$, be a diffusion process over some metric space $E$. Let $\partial$ be the extra point in the one-point compactification of $E$. Usually $\partial$ is regarded as the cemetery of the process. For every $x\in E$, let $\P_x$ denote the law of $X_t$ starting from $x$. Let $\rho$ be a nonnegative function on $E$ such that $\rho(x)$ tends to infinity as $x$ wanders out to $\partial$. If there is a non-decreasing positive function $R(t)$ on $\R_+$ such that 
\[
    \P_x(\rho(X_t)\leq R(t)\ \mathrm{for\ all\ sufficiently\ large}\ t)=1, 
\]
then we call $R(t)$ an upper rate function or an upper radius for $X_t$ with respect to $\rho$.
It is an upper bound of the escape rate of $X_t$. 

As an example, for every
$\varepsilon>0$, $(1+\varepsilon)\sqrt{2t\log\log t}$ is an upper rate function for
the Brownian motion $B_t$ 
on a Euclidean space with respect to the Euclidean metric. 
This follows easily 
from the celebrated Khinchin law of iterated logarithm which 
 asserts that 
\begin{equation}\label{ILL}
\varlimsup_{t\to \infty} \frac{|B_t|}{\sqrt{2t\log\log t}}=1
\end{equation}
almost surely. 

The aim of this paper is to study upper rate functions for diffusion processes. 
Note that it is a classical question to ask how fast a diffusion process wanders out to the cemetery. For instance,  \cite{Fri06, GS72, IM74,Mao08,Mot59,SW73} studied the asymptotic behavior of diffusion process. In particular, \cite{Gri99_AA, GH09, HQ10} studied the escape rate of Brownian motion on Riemannian manifold. 

A related problem is the conservativeness of diffusion process, see e.g. \cite{Gri99_Bull, Stu94_JRAM, Tak89, Tak91} and references therein. It is clear that if we get a finite upper bound of the escape rate of the process, then it follows immediately that the process is conservative. 

Intuitively, if a process gets much free space to live on then it can run to the cemetery faster. So it is natural to characterize escape rate in terms of volume growth of the underlying state space. Indeed, \cite{GH09,HQ10} obtained upper rate functions for Brownian motions on Riemannian manifolds in terms of the volume growth of concentric balls.


It is well known that local Dirichlet form is an appropriate frame to unify and extend results for Brownian motion on a complete Riemannian manifold. For instance,  it is in this spirit that \cite{BS05} and \cite{Stu95_PrPr} generalized the results in \cite{Gri99_AA} and \cite{Gri99_Bull} respecitively.  
The first aim of this paper is to extend the main result in \cite{HQ10} from the Riemmanian setting to the framework of Dirichlet form.  More precisely, we give an upper rate function with respect to a nonnegative function $\rho$ on the state space for the symmetric diffusion process associated with a regular, local and symmetric Dirichlet form (cf. Theorem \ref{Theoem_rate}). The upper rate function is  in terms of the volume growth of the underlying state space and the growth of the energy density $\Gamma(\rho,\rho)$ of the function $\rho$.

It is canonical to take $\rho$ as the intrinsic metric associated with the Dirichlet form. Then,  under some topological condition,  the energy density  $\Gamma(\rho,\rho)$ is bounded above by 1 and hence the volume growth of the metric ball in the state space is the the unique condition.  Note that the volume of a metric ball involves exactly two of the most essential ingredients of Dirichlet space, i.e. the measure on the state space and the metric induced intrinsically from the Dirichlet form.  However, without additional geometric condition, one cannot always expect to get optimal upper rate function in terms of the volume growth condition only. For example, for Brownian motion on Euclidean space, using volume growth condition only, we can obtain upper rate function of order $\sqrt{t\log t}$ which is rougher than the consequence induced by the iterated logarithm law \eqref{ILL}. 

So in the second part of this paper, we turn to consider geometric condition for the  diffusion process with infinitesimal generator $L$ on a Riemannian manifold $M$. Let $\rho_o$ be the distance function with respect to some fixed point  $o$ on $M$.
 Suppose that there is a measurable function $\theta$ on $[0,+\infty)$ such that 
\begin{equation*}\label{Equ:L_rho<theta:sec:intro}
L\rho_o\leq \theta(\rho_o)
\end{equation*}
on $M\setminus \mathrm{cut}(o)$. Here  $\mathrm{cut}(o)$ is the cut locus of $o$. 
Under this condition, we establish a comparison theorem (cf. Theorem \ref{Thm:SemigroupComp}) to compare the radial process of the diffusion process  with respect to a scaled Brownian motion with drift $\theta$. Then upper rate function for the diffusion process on manifold is controlled by the upper rate function for the solution process of a one dimensional stochastic differential equation. 

This paper is organized as follows. 
In Section \ref{Sec2} we extend the main result in \cite{HQ10}. Then in Sections \ref{Sec:examples_Dirichlet} we give some examples to illustrate how to apply the extended result. 
In Section \ref{Sec:Comparison-Thm} we prove a comparison theorem for escape rate.  In Section \ref{Sec:escape-one-dim} we study briefly upper rate functions for one dimensional It\^o diffusion processes. 

\section{Escape rate of Dirichlet process}\label{Sec2}
Consider a measure space $(E,\B, \m)$, where 
$E$ is a locally compact separable Hausdorff space, and $\m$ is a positive Radon measure on $E$ with full support. Let $\H=L^2(E,\m)$ be the Hilbert space of square integrable (with respect to $\m$) extended real valued functions on $E$
endowed with the usual inner product $\<\cdot,\cdot\>$ and norm $\|\cdot\|_2$. Let  $\C(E)$ be the space of continuous functions on $E$ and
$\C_0(E)$ the space of all continuous functions with compact support on $E$. 
Let $\partial$ be the extra point in the one-point compactfication of $E$: $E\cup {\partial}$ is compact. 

We consider a symmetric regular and strongly local Dirichlet form  $(\EC,\F)$  with definition domain $\F$ on $\H$. Let us recall these notions shortly (see \cite{FOT94} for more details). We call $\EC$ symmetric if $\EC(f,g)=\EC(g,f)$ for all  $f,g\in \F$.
The Dirichlet form $\EC$ is called regular if it posseses a core in $\F$. A core is subset of $\F\cap \C_0(E)$ that is dense in $\C_0(E)$ for the uniform norm and dense in $\F$ for the norm $\sqrt{\|f\|_2^2+\EC(f,f)}$. 
$\EC$ is called strongly local if $\EC(f,g)=0$ for every $f,g\in \F$ with compact supports satisfying the condition that $g$ is constant on a neighborhood of the support of $f$. A regular Dirichlet form is strongly local if and only if both the jumping measure and the killing measure of the Dirichlet form vanish.

It is well known that (cf. \cite{BH91,FOT94}) for any $f,g\in \F$, there is a signed Radon measure $\m_{\<f,g\>}$, the energy measure of $f$ and $g$ on $E$, such that
\[
\EC(f,g)= \int_E d \m_{\<f,g\>}, \quad f,g\in \F. 
\]
Obviously, $\m_{\<f,g\>}$ is a positive semidefinite symmetric bilinear form in $f,g$ and it enjoys locality property. Hence the definition domain $\F$ of the energy measure can be extended to $\F_\loc$. Here $\F_{\loc}$ is the set of all $\m$-measurable functions $f$ on $E$ for which on every relatively compact open set $\Omega\subset E$, there exists a function $f'\in \F$ such that $f=f'$ $\m$-a.s. on $\Omega$. 

We shall assume the following assumption. 

\begin{assumption}\label{Assumption_rho}
Suppose that there is a  nonnegative
function $\rho$ on $E$ satisfying the following conditions: 
\begin{enumerate}
 \item  $\rho\in \F_\loc\cap \C(E)$. 
 \item \(\lim_{x\to \partial}\rho(x)=+\infty\).
   \item For any $r>0$, the ball $B(r):=B_\rho(r):=\{x\in E\colon \rho(x)\leq r \}$ is compact. 
  \item The energy measure $\m_{\<\rho,\rho\>}$ of $\rho$ is absolutely continuous with respect to $\m$.  
\end{enumerate}
\end{assumption}

When the energy measure $\m_{\<\rho,\rho\>}$ of $\rho$ is absolutely continuous with respect to $\m$, we set 
$$\Gamma(\rho,\rho)= \frac{d \m_{\<\rho,\rho\>}}{d \m}$$
and call it the energy density of $\rho$ with respect to $\m$. 
For any $r>0$, let 
       \[
            \lambda(r):=\lambda_\rho(r):=\sup_{y\in B(r)} \Gamma(\rho,\rho)(y)   .     
       \]
It desribes the growth of the energy density $\Gamma(\rho,\rho)$.

Associated with the Dirichlet space $(\EC, \F)$, there is a canonical diffusion process $(\PP(E), X_t, 0\leq t< \zeta,\P_x)$. Here  $\PP(E)$ is the path space over $E$ (i.e., all continuous real valued functions on $[0,\zeta)$), $X_t$ is the coordinate process defined via $X_t(\omega)=\omega(t)$ for all $\omega\in \PP(E)$ and $0\leq t< \zeta$, $\zeta$ is the life time of $X_t$, and $\P_x$ is the law of $X_t$ with initial point $x\in E$. Recall that $(\EC, \F)$ is called conservative if 
	$$\P_x(\zeta=\infty)=1$$
for all $x\in E$. 

Now we are ready to state the following result on the escape rate of $X_t$. 
As indicated in Example \ref{Classical-BM},  this is a generalization of \cite[Theorem 4.1]{HQ10}.
\begin{theorem}\label{Theoem_rate}
Suppose that there is a nonnegative function $\rho$ on $E$
 satisfying Assumption \ref{Assumption_rho}. 
Let $x\in E$ and  $\psi$ be a function on $[0,+\infty)$ determined by 
\begin{equation}\label{volume_condition}
  t=\int_{2}^{\psi(t)}
         \frac{r}{\lambda(r)\left(\log \m(B(r))+\log\log r\right)}
  \,dr. 
\end{equation}
Then there exists a constant $C$ such that $\psi(Ct)$ is an upper rate function for $X_t$ with respect to $\rho$, that is
\begin{equation}\label{Equ:rate_fnt_psi}
   \P_x(\rho(X_t)\leq \psi(Ct)\ \mathrm{for\ all\ }t\ \mathrm{sufficiently\ large})=1.
\end{equation}
\end{theorem}

Consequently, we have the following criterion for conservativeness of the diffusion process. 

\begin{cora}\label{Coro-explosion}
Let $\rho$ be a nonnegative function on $E$ satisfying Assumption \ref{Assumption_rho}. 
If \begin{equation}\label{Equ:explosion-cond}
     \int_{2}^{\infty}
         \frac{r}{\lambda(r)\left(\log \m(B(r))+\log\log r\right)}
  \,dr=\infty, 
\end{equation}
then the Dirichlet form $(\EC,\F)$ is conservative. 
\end{cora}

There is a canonical  choice of the function $\rho$ which satisfies Assumption \ref{Assumption_rho}. Let 
\[
   d(x,y):=\sup\{ f(x)-f(y)\colon f\in \F\cap \C_0(E),
                 d\m_{\<f,f\>} \leq d\m \ \mathrm{on}\ E \}
\]
for every $x,y\in E$.  
Here  $d\m_{\<f,f\>} \leq d\m$ means that the energy measure $d\m_{\<f,f\>}$ is absolutely continuous with respect to the reference measure $\m$ and the Radon-Nikodym derivative $$\frac{d \m_{\<f,f\>}}{d\m}\leq 1$$
$\m$-a.s. on $E$.   

The function $d$ is called  intrinsic metric. We refer to \cite{BM91, BM95, Stu94_JRAM, Stu95_PrPr, Stu98_NDDF} etc. for the details. 
For any $x\in E$, let 
$$\rho_x(z):=d(z,x)\quad \mathrm{for\ all}\ z\in E.$$
We need the following assumption such that 
$\rho_x$ satisfies Assumption \ref{Assumption_rho}.

\begin{assumption}\label{Top_Assumption}
 The intrinsic metric $d$ is a complete metric on $E$ and the topology induced by $d$ is equivalent to the original topology on $E$. 
\end{assumption}

By Assumption \ref{Top_Assumption},  we have (see \cite[Lemma 1$'$]{Stu94_JRAM})  $\rho_x\in \F_\loc\cap \C(E)$ and 
\begin{equation}\label{gradient_prop}
\Gamma(\rho_x,\rho_x)\leq 1.
\end{equation}
Moreover, $B_{\rho_x}(r)$ is compact for every $r>0$.  So Assumption \ref{Assumption_rho} holds with $\lambda(r)\leq 1$.  Thus  Theorem \ref{Theoem_rate} implies the following corollary.

\begin{cora}\label{cora_of_Theoem_rate}
Suppose that Assumption \ref{Top_Assumption} holds. 
Let $x\in E$ and $\psi$ be a function on $[0,+\infty)$ determined by 
\begin{equation}\label{volume_condition}
  t=\int_{2}^{\psi(t)}
         \frac{r}{\log \m(B_{\rho_x}(r))+\log\log r}
  \,dr. 
\end{equation}
Then there exists a constant $C$ such that $\psi(Ct)$ is an upper rate function for $X_t$ with respect to $\rho_x$. That is,  
\begin{equation}\label{Equ:rate_fnt_psi}
   \P_x(\rho_x(X_t)\leq \psi(Ct)\ \mathrm{for\ all\ }t\ \mathrm{sufficiently\ large})=1.
\end{equation}
\end{cora}

\begin{remark}
In terms of the intrinsic metric, Corollary \ref{Coro-explosion} 
 naturally gives a probabilistic proof of Sturm's conservativeness test \cite[Theorem 4]{Stu94_JRAM}: $(\EC,\F)$
is conservative if for some $x\in E$, 
\begin{equation}\label{explode_test}
   \int_2^\infty \frac{r}{\log \m (B{\rho_x}(r))}\,dr=\infty.
\end{equation} 
Note that in this case the extra term $\log\log r$ appeared in \eqref{Equ:explosion-cond} can be dropped (see  \cite[Lemma 3.1]{HQ10} for a proof). 
\end{remark}

Before proceeding to the proof of Theorem \ref{Theoem_rate}, 
let us explain the idea of the proof.

Let $\{R_n\}_{n\geq 1}$ be a sequence of strictly increasing radii to be determined later such that $R_n\uparrow +\infty$ as $n\uparrow +\infty$ and set $B_n:=B(R_n)$ for every $n\geq 1$. 
Suppose that the process starts from points in $B_1$. Consider  the first exit time of $X_t$ from $B_n$:
\[
   \tau_n:=\inf\{t>0\colon X_t\notin B_n\},\quad n\geq 1.
\]
For every $n\geq 1$, it is clear that at the stopping time $\tau_n$, the process first reaches the boundary $\partial B_n$. Let $\tau_0=0$. Then for every $n\geq 1$,  the difference $\tau_n-\tau_{n-1}$, provided that $\tau_{n-1}<+\infty$ almost surely, is the crossing time of the process from $\partial B_{n-1}$ to $\partial B_n$.

Suppose that we have a sequence of time steps $\{t_n\}_{n\geq 1}$ such that
\[
      \sum_{n=1}^\infty\P_x (\tau_n-\tau_{n-1} \leq t_n)<\infty.
\]
Then by the Borel-Cantelli's lemma, the probability that the events $\{\tau_n-\tau_{n-1} \leq t_n\}$ happen infinitely often is 0. So, for all large enough $n$ we have
\[
  \tau_n-\tau_{n-1} > t_n,\quad \P_x\textrm{-}\mathrm{a.s.}.
\]
Roughly speaking,  it implies 
\[
    T_n:=\sum_{k=1}^n t_k<\tau_n, \quad \P_x\textrm{-}\mathrm{a.s.}. 
\]
It indicates that almost surely the process stays in $B_n$ before time $T_n$. This connection will give us an upper rate function for $X_t$ after some manipulation.  

Now the problem is reduced to estimates of the crossing times $\tau_n-\tau_{n-1}$.  To get the estimates, analytic and probabilistic approaches  are used in \cite{GH09} and \cite{HQ10} respectively. The main idea in \cite{HQ10} is to use Lyons-Zheng's decomposition (\cite{LZ88}). This goes back to \cite{Tak89,Tak91} where conservativeness of general symmetric diffusion process was studied. Note that some additional geometric condition is required by the analytic approach in \cite{GH09}. 

In this paper we adopt the procedure in \cite{HQ10}. Let us first define some notation. For any compact set $K\subset E$,  set
\begin{equation}\label{Def:equ:P_K}
   \P_K=\frac{1}{\m(K)}\int_K \P_z \,\m(dz)
\end{equation}
and
\[
   \EC^K(f,g)=\int_K \Gamma(f,g)\,d\m,\quad f,g\in\F.
\]
It is clear that $\EC^K$ is closed on $L^2(X,\Eins_K \m)$. Let $\F^K$ be the domain of the closure of $\EC^K$. Then $(\EC^K,\F^K)$ is a strongly local,  regular, symmetric and conservative Dirichlet form on $L^2(K,\Eins_K \m)$. 

As remarked in \cite{HQ10},  to use the Lyons-Zheng's decomposition, we have to consider the following events
\[
   A_n:=\{\tau_n-\tau_{n-1}\leq t_n,\ \tau_n\leq T_n\},\quad n\geq 1
\] 
instead of the events $\{\tau_n-\tau_{n-1}\leq t_n\}$, $n\geq 1$.

The estimate in the following lemma is crucial for the proof of Theorem \ref{Theoem_rate}. 

\begin{lemma}\label{lemma:crosstime}
Let $R_0=0$ and $r_n:=R_n-R_{n-1}$ for every $n\geq 1$. 
Then for every $n\geq 1$, 
\begin{equation}\label{Equ:P_B1_A_n-estimate} 
\P_{B_1}(A_n)\leq \frac{16}{\sqrt{2\pi}} \frac{\m(B_n)}{\m(B_1)}\frac{T_n\sqrt{\lambda(R_n)}}{\sqrt{t_n}r_n}\exp\left(-\frac{r_n^2}{8\lambda(R_n) t_n}\right).
\end{equation}
\end{lemma}
\begin{proof}
Set $\rho_t:=\rho(X_t)$, $\tau_0=0$ and  $\tau_n:=\inf\{t>0\colon X_t\notin B_n\}$ for every $n\geq 1$. 
For every $r>0$, 
let $\M^r=(X_t,\P_x^r)$  be the diffusion processes corresponding to  $(\EC^{B(r)},\F^{B(r)})$. 

Analogous to \eqref{Def:equ:P_K}, we define for every compact set $K\subset B(r)\subset E$,
 \begin{equation}\label{Def:equ:P_K'}
   \P_K^r=\frac{1}{\m(K)}\int_K \P_z^r \,\m(dz).
\end{equation}
Then
  \begin{equation}\label{eq1:pf:lemma:crosstime}
    \begin{aligned}
    \P_{B_1}(A_n)=   & \P_{B_1}(\tau_n-\tau_{n-1}\leq t_n,\, \tau_n\leq T_n) \\
   \leq& \P_{B_1} \left ( \sup_{0\leq s\leq t_n} (\rho_{\tau_{n-1}+s}-\rho_{\tau_{n-1}})
                         \geq r_n, \, \tau_n\leq T_n \right)  \\
   \leq& \frac{\m(B_{n-1})}{\m(B_1)}
			\P_{B_{n-1}} \left( \sup_{0\leq s\leq t_n} (\rho_{\tau_{n-1}+s}-	\rho_{\tau_{n-1}})
                         \geq r_n, \, \tau_n\leq T_n\right)  \\
%
\leq & \frac{\m(B_{n-1})}{\m(B_1)}
			\P_{B_{n-1}}^{R_{n}} \left( \sup_{0\leq s\leq t_n} (\rho_{\tau_{n-1}+s}-	\rho_{\tau_{n-1}})
                         \geq r_n, \, \tau_n\leq T_n\right)  \\
\leq & \frac{\m(B_n)}{\m(B_1)}
			\P_{B_n}^{R_n} \left( \sup_{0\leq s\leq t_n} (\rho_{\tau_{n-1}+s}-\rho_{\tau_{n-1}})
                         \geq r_n, \, \tau_n\leq T_n\right) .
    \end{aligned}
  \end{equation}
Since the diffusion process $\M^n$ is conservative, by Lyons-Zheng's decomposition (see \cite{LZ88}) we have
\begin{equation}\label{LZ-Decomp}
   \rho_t-\rho_0 =\frac12 M_t - \frac12(\tilde{M}_{T_n} - \tilde{M}_{T_n-t}),\quad
                   \P_{B_n}^{R_n}\textrm{-}\mathrm{a.s.} .
\end{equation}
Here $M_t$ is a martingale additive functional of finite energy (see the notation in \cite{FOT94}) and 
$$\tilde{M}_t=M_t(r_{T_n}),$$ where $r_{T_n}$ is the time reverse operator at $T_n$ defined via 
   $$X_t(r_{T_n})=X_{T_n-t}.$$
Set 
$$\F_t=\sigma(X_s\colon 0\leq s\leq t)
\quad \mathrm{and} 
\quad \G_t=\sigma(X_s\colon T_n-t\leq s\leq T_n).$$ 
Then $(M_t,\P_{B_n}^{R_n})$ is a $\F_t$-martingale, while $(\tilde{M}_t,\P_{B_n}^{R_n})$ is a  $\G_t$-martingale. 

As discussed in \cite{HQ10}, by \eqref{LZ-Decomp} we have
\begin{equation}\label{eq2:pf:lemma:crosstime}
\begin{aligned}
       & \left\{ \sup_{0\leq s\leq t_n} (\rho_{\tau_{n-1}+s}-\rho_{\tau_{n-1}} ) \geq r_n,\tau_n\leq T_n
            \right\} \\
\subseteqq & \bigcup_{k=1}^{\left[\frac{T_n}{t_n}\right]+1} 
            \left\{ \sup_{|s|\leq t_n} |\rho_{kt_{n}+s}-\rho_{kt_{n}} | 
                \geq \frac{r_n}{2} \right\} \\
\subseteqq &  \bigcup_{k=1}^{\left[\frac{T_n}{t_n}\right]+1} 
            \left(
            \left\{ \sup_{|s|\leq t_n} |M_{kt_{n}+s}-M_{kt_{n}} | 
                \geq \frac{r_n}{2} \right\}                 
            \bigcup 
                  \left\{ \sup_{|s|\leq t_n} |\tilde{M}_{kt_{n}+s}-\tilde{M}_{kt_{n}} | 
                \geq \frac{r_n}{2} \right\}    \right)
\end{aligned}
\end{equation}

From the symmetry of $M_t$, we have the time reversibility: For any $\F_{T_n}$-measurable function $F$, 
	$$\E_{B_n}^{R_n}(F(r_{T_n}))=\E_{B_n}^{R_n}(F).$$ 
Here 	$\E_{B_n}^{R_n}$ is the expectation with respect to $\P_{B_n}^{R_n}$.
Hence for every $1\leq k\leq \left[\frac{T_n}{t_n}\right]+1$, 
\begin{equation}\label{Equ:same-prob-M-M-rt}
\begin{aligned}
   &\P_{B_n}^{R_n} \left( \sup_{|s|\leq t_n} |M_{kt_{n}+s}-M_{kt_{n}} | 
                \geq \frac{r_n}{2}  \right) \\
=& \P_{B_n}^{R_n}         
\left ( \sup_{|s|\leq t_n} |\tilde{M}_{kt_{n}+s}-\tilde{M}_{kt_{n}} | 
                \geq \frac{r_n}{2} \right).        
\end{aligned}                
\end{equation}
Therefore, 
it follows from \eqref{eq1:pf:lemma:crosstime},
\eqref{eq2:pf:lemma:crosstime} and
\eqref{Equ:same-prob-M-M-rt} that 
\begin{equation}\label{Equ:P_B1}
\begin{aligned}
     & \P_{B_1}(\tau_n-\tau_{n-1}\leq t_n,\, \tau_n\leq T_n) \\
\leq &  2 \frac{\m(B_n)}{\m(B_1)} 
\sum_{k=1}^{\left[\frac{T_n}{t_n}\right]+1}
            \P_{B_n}^{R_n} \left( \sup_{|s|\leq t_n} |M_{kt_{n}+s}-M_{kt_{n}} | 
                \geq \frac{r_n}{2}  \right).
\end{aligned}
\end{equation}
So it remains to estimate the probability 
$$\P_{B_n}^{R_n} \left( \sup_{|s|\leq t_n} |M_{kt_{n}+s}-M_{kt_{n}} | 
                \geq \frac{r_n}{2}  \right)$$ for every $1\leq k\leq \left[\frac{T_n}{t_n}\right]+1$. 
                
 Note that continuous martingale $M_t-M_0$ is a time change of one dimensional standard Brownian motion $B(t)$ with respect to $\P_{B_n}^{R_n}$. That is, for every $t\geq 0$, 
\[
	M_t-M_0=B(\<M\>_t)=B\left(\int_0^t\Gamma(\rho,\rho)(X_u)\,du\right).
\]
Since $\Gamma(\rho,\rho)\leq \lambda(R_n)$ on $B_n$,  we have
for every $1\leq k\leq \left[\frac{T_n}{t_n}\right]+1$,
\begin{equation}\label{Equ:estimage_P_Bn-M-M}
\begin{aligned}
   & \P_{B_n}^{R_n} \left( \sup_{|s|\leq t_n} |M_{kt_{n}+s}-M_{kt_{n}} | 
                \geq \frac{r_n}{2}  \right) \\
= &   \P_{B_n}^{R_n}    
     \left(   
	 \sup_{|s|\leq t_n} 	 
	 \left|B\left(\int_{kt_{n}}^{kt_{n}+s}
	                    \Gamma(\rho,\rho)(X_u)\,du
	           \right)\right|             
                           \geq \frac{r_n}{2}
      \right) \\      
 \leq  & 2 \P_{B_n}^{R_n}     \left(   
         \sup_{0\leq s \leq t_n} 	 
	         \left|B(\lambda(R_n) s)\right|             
                           \geq \frac{r_n}{2}
      \right)   \\
\leq & 4  \P_{B_n}^{R_n} \left( \frac{1}{\sqrt{\lambda(R_n) t_n}} B( \lambda(R_n) t_n)\geq \frac{r_n}{2\sqrt{\lambda(R_n) t_n}} \right) \\
= & 4\frac{1}{\sqrt{2\pi}} \int_{\frac{r_n}{2\sqrt{\lambda(R_n) t_n}}}^\infty \exp\left(-\frac{x^2}{2}\right) dx \\
\leq & \frac{8}{\sqrt{2\pi}} \frac{\sqrt{\lambda(R_n) t_n}}{r_n} \exp\left(-\frac{r_n^2}{8\lambda(R_n) t_n}\right).         
\end{aligned}                
\end{equation}
Here we have used the following simple inequality 
\[
	\int_a^\infty \exp\left(-\frac{x^2}{2}\right) dx \leq \frac1a 
	\exp\left(-\frac{a^2}{2}\right),\quad a>0.
\]
Now \eqref{Equ:P_B1_A_n-estimate} follows immediately from 
 \eqref{Equ:P_B1} and \eqref{Equ:estimage_P_Bn-M-M}.
Thus the proof is complete. 
\end{proof}

Now let us consider how to choose $R_n$ and $t_n$ properly. 

By $\sqrt{x}<\e^x$ for all $x>0$, we have 
\[
\frac{1}{\sqrt{t_n}}\leq \frac{\sqrt{16\lambda(R_n)}}{r_n}
\exp\left(
		\frac{r_n^2}{16\lambda(R_n)t_n}
		\right).
\]
So it follows from 
\eqref{Equ:P_B1_A_n-estimate}
that 
\begin{equation}\label{Equ:P_B1_A_n-estimate-new1} 
\P_{B_1}(A_n)\leq \frac{64}{\sqrt{2\pi}} \frac{1}{\m(B_1)}\frac{T_n\lambda(R_n) }{r_n^2}\exp\left(\log \m(B_n)-\frac{r_n^2}{16\lambda(R_n) t_n}\right).
\end{equation}
In order to absorb the coefficient of the exponential part, it is natural to take 
\begin{equation*}
    t_n=\frac{r_n^2}{32\lambda(R_n)(\log \m(B_n)+h( R_n))}, 
\end{equation*}
where  $h$ is an increasing function to be  determined later.
By \eqref{Equ:P_B1_A_n-estimate-new1}
we get 
\begin{equation}\label{Equ:P_B1_A_n-estimate-new2} 
\P_{B_1}(A_n)\leq \frac{64}{\sqrt{2\pi}} \frac{1}{\m(B_1)}\frac{T_n\lambda(R_n) }{r_n^2}\exp\left(-\log \m(B_n)-2h(R_n)\right).
\end{equation}
Suppose that $r_n$ is increasing in $n$. 
Then 
\[
\begin{aligned}
	T_n=\sum_{k=1}^{n}t_k 
	 = &
		\frac{1}{32\lambda(R_n)\left(\log \m(B_n)+h( R_n)\right)}
		\sum_{k=1}^{n}r_k^2  \\
	 \leq &
		\frac{r_n}{32\lambda(R_n)\left(\log \m(B_n)+h( R_n)\right)}
		\sum_{k=1}^{n}r_k \\
	= & \frac{R_n r_n}{32\lambda(R_n)\left(\log \m(B_n)+h( R_n)\right)}	   
		   .
\end{aligned}	
\]
Substituting the above inequality into \eqref{Equ:P_B1_A_n-estimate-new2}, we get 
\begin{equation}\label{Equ:P_B1_A_n-estimate-new3} 
\P_{B_1}(A_n)\leq 
\frac{2}{\sqrt{2\pi}} \frac{1}{\m(B_1)}
\frac{1}{\log \m(B_n)+h( R_n)}
 \frac{R_n}{r_n}\exp\left(-2h(R_n)\right).
\end{equation}
To obtain  $\sum_{n=1}^\infty \P_{B_1}(A_n)<\infty$ 
it suffices to take the Radii $R_n$ and the function $h$ such that 
\begin{equation}\label{Equ:aim-est2:sum_R/r}
	\sum_{n=1}^\infty \frac{R_n}{r_n}\exp\left(-2h(R_n)\right)<\infty.
\end{equation}
Let $R_n=2^{n}c$ for some large enough constant $c>0$ such that $x\in B_1$. Since $R_n/r_n=2$,  \eqref{Equ:aim-est2:sum_R/r} is equivalent to 
\[
		\sum_{n=1}^\infty \exp\left(-2h(2^{n}c)\right)<\infty. 
\]
Evidently it is sufficient to take $h(r)=\log\log r$ since in this case we have 
$$\exp\left(-2h(2^n c)\right)=(n\log 2+\log c)^{-2}\approx \frac{1}{n^2}.$$ 

We are now in a position to prove Theorem  \ref{Theoem_rate}.

\begin{proof}[Proof of Theorem \ref{Theoem_rate}]
According the discussion above, by taking $R_n=2^n c$ for some large enough constant $c>0$ and
\begin{equation}\label{Equ:t_n-def}
    t_n=\frac{r_n^2}{32\lambda(R_n)(\log \m(B_n)+\log\log R_n)},
\end{equation}
we get $\sum_{n=1}^\infty\P_{B_1}(A_n)<\infty$. As shown in \cite[Lemma 2.1]{HQ10}, there exists a constant $T_{-1}\geq 0$ such that for all $n\geq 1$, 
$$\tau_n\geq T_n-T_{-1},\quad \P_{B_1}\textrm{-}\mathrm{a.s.}.$$
Hence 
\begin{equation}\label{Ineq:sup_rho_t-leq-2^n1} 
	\sup_{t\leq T_n-T_{-1}} \rho_t\leq 2^n c, \quad \P_{B_1}\textrm{-}\mathrm{a.s.}.
\end{equation}
Note that 
\[
	r_n^2=
	\frac14 R_n(R_{n+1}-R_n).
\]
By \eqref{Equ:t_n-def} we have 
\begin{equation}\label{Equ:Tn_control}
\begin{aligned}
	T_n
	  &=\sum_{k=1}^n t_k=
	  \sum_{k=1}^n
  	\frac{R_k(R_{k+1}-R_k)}{128\lambda(R_k)(\log \m(B_k)+\log\log R_k)}
	 \\
	&\geq \frac1{256}\int_{R_1}^{R_{n+1}}
	   \frac{r}{\lambda(r)(\log \m(r)+\log\log r)}dr
	.
\end{aligned}	
\end{equation}
Let
\begin{equation}\label{volume_condition-invers}
 \phi(R)=\int_{2c}^{R}
         \frac{r}{\lambda(r)\left(\log \m(B(r))+\log\log r\right)}
  \,dr. 
\end{equation}
Then \eqref{Equ:Tn_control} implies  
\begin{equation}\label{Equ:T_n-geq-1/156-phi}
T_n\geq \frac1{256} \phi(2^{n+1}c).
\end{equation}
For every $R\geq 2c$, let $n(R)$ 
be the positive integer 
such that 
\begin{equation}\label{Equ:n_r-c}
2^{n(R)}c< R\leq 2^{n(R)+1}c.
\end{equation}
As $\phi$ is increasing, by \eqref{Equ:T_n-geq-1/156-phi} and
\eqref{Equ:n_r-c}
 we have  
\[\frac1{256}\phi(R)\leq  \frac1{256} \phi(2^{n(R)+1}c) \leq T_{n(R)}.\]
Therefore, by  \eqref{Ineq:sup_rho_t-leq-2^n1} and \eqref{Equ:n_r-c},  for all $R\geq 2c$
\begin{equation}\label{Ineq:sup_rho_t-leq-2^n2} 
	\sup_{t\leq \frac1{256}\phi(R)-T_{-1}} \rho_t\leq 2^{n(R)} c<R,
	 \quad \P_{B_1}\textrm{-}\mathrm{a.s.}.
\end{equation}
Let $\tilde{\psi}$ be the inverse function of $\phi$. Then 
\eqref{Ineq:sup_rho_t-leq-2^n2} implies 
\begin{equation}\label{Ineq:sup_rho_t-leq-2^n3} 
	\sup_{t\leq T} \rho_t\leq \tilde{\psi}(256(T+T_{-1})), 
	\quad \P_{B_1}\textrm{-}\mathrm{a.s.}
\end{equation}
for large enough $T>0$. 
Consequently, in terms of $\psi$, we have for sufficiently large $T>0$, 
\begin{equation}\label{Ineq:sup_rho_t-leq-2^n4} 
\sup_{t\leq T} \rho_t\leq \psi(512T) ,
\quad \P_{B_1}\textrm{-}\mathrm{a.s.}.
\end{equation}

Now we arrive at the conclusion that 
$$\P_{B_1}(H)=1,$$
 where 
\[
   H=\{\rho(X_t)\leq \psi(Ct)\ \mathrm{for\ all}\ t\ \mathrm{sufficiently\  large}\}.   
\]
This means that the function $\psi$ is an upper rate function for $X_t$ starting from points with uniform distribution on the ball $B_1$. The theorem will proved by showing  that $\psi$ is also an upper rate function for $X_t$ starting from every single point in $B_1$. 

By Markov property of $X_t$, we have for every $z\in E$, 
\begin{equation}\label{Equ:Markov-property-1H}
	\E_z(\Eins_H\circ \theta_t|\F_t)=\E_{X_t}\Eins_H,
\end{equation}
where $\theta_t$ is the shift operator on $\PP(E)$ defined by 
$$(\theta_t\omega)(s)=\omega(t+s),\quad \omega\in \PP(E).$$
 Obviously we have  $\Eins_H\circ \theta_t=\Eins_H$. Hence from \eqref{Equ:Markov-property-1H} we get 
\begin{equation}\label{Equ:Ez-EX_t}
	\E_z(\Eins_H|\F_t)=\E_{X_t}\Eins_H.
\end{equation}
Let \[h(z):=\P_z(H)=\E_z\Eins_H\] for all $z\in E$. 
By \eqref{Equ:Ez-EX_t} and the definition of $h$, we obtain 
\[
	P_th(z)=\E_z h(X_t) = \E_z(\E_{X_t}\Eins_H)
	           =\E_z(\E_z(\Eins_H|\F_t)).
\]
Note that $\E_z(\E_z(\Eins_H|\F_t))
	           =\E_z\Eins_H =h(z)$. So we have 
$$P_th(z)=h(z)$$ for all $z\in E$.  	           
This proves that $h$ is a $L$-harmonic function on $E$ (i.e. solution of $Lu=0$, cf.  \cite{Stu94_JRAM}), where $L$ is the infinitesimal generator associated with $X_t$. 
By Liouville theorem (cf. \cite{Stu94_JRAM}), the function $h$ is constant on $B_1$. On the other hand, we have 
\[
      \P_{B_1}(H)=\frac{1}{\m(B_1)}\int_{B_1} \P_z(H) \,d\m
      =\frac{1}{\mu(B_1)}\int_{B_1} h(z)\,\m(dz)=
           1. 
\]
Therefore we must have $\P_z(H)=h(z)=1$ for  every $z\in B_1$. In particular, we get \eqref{Equ:rate_fnt_psi} and hence the proof is now complete. 
\end{proof}

\begin{proof}[Proof of Corollary \ref{Coro-explosion}]
From \eqref{Equ:rate_fnt_psi}, we have for 
any $x\in E$, there is a constant $C>0$ such that 
for all sufficiently large $T>0$, 
$$
	\sup_{t\leq T} \rho_t\leq \psi(CT), \quad \P_x\textrm{-}\mathrm{a.s.}.
$$
Hence we have for any sufficiently large $T>0$, 
\[
	\P_x(T<\zeta)=\P_x\left(\sup_{t\leq T} \rho_t<\infty\right)=1.
\]
 This proves that the process is conservative. 
\end{proof}

 
 \section{Examples of escape rate of Dirichlet process}
 \label{Sec:examples_Dirichlet}
\begin{example}\label{Classical-BM}
Let  $(M,g)$ be a $n$-dimensional complete connected Riemannian manifold with Riemannian metric $g$. Let $\varDelta$ and $\nabla$ be the Laplace-Beltrami operator and  gradient operator on $M$ respectively.  Denote the inner product in the tangent space $T_x M$ at $x\in M$ by $\<\cdot,\cdot\>:=\<\cdot,\cdot\>_x:=g_x(\cdot,\cdot)$. Let $d$ be the Riemannian distance function on $M$ and $\rho_o(\cdot)=d(o,\cdot)$ the distance function on $M$ with respect to some fixed point $o\in M$. Let $B(r):=\{x\in M\colon \rho_o(x)\leq r\}$ be the ball with center $o$ and radius $r> 0$. Let $\mathrm{vol}(dx)$ be the volume element of the manifold $M$, and  $TM$  the bundle of tangent space of $M$. 

Let $A\colon TM\to TM$ be a strictly positive definite mapping and $V$ a smooth function on $M$. Set  $\m(dx)=\exp(V(x))\mathrm{vol}(dx)$.
Due to the integration by parts formula, 
\[
   \EC(f,g)=\int_M \<A\nabla f,\nabla g\>\,\mu(dx),\quad f,g\in \C_0^1(M),
\]
defines a closable Markovian form on $L^2(M,\mu(dx))$. Its closure $(\EC,\F)$  is a strongly local,  regular, conservative and symmetric Dirichlet form on $L^2(M,\mu(dx))$. 

In the case when $A$ is the identity operator on $TM$ and $V$ vanishes, $(\EC,\F)$ is the classical Dirichlet form. The associated intrinsic metric coincides with the Riemannian distance and Assumption \ref{Assumption_rho} holds for $\rho_o$.  The corresponding diffusion process is the Brownian motion $B_t$ on $M$. By Corollary \ref{Coro-explosion}, there exists some constant $C>0$ such that $\psi(Ct)$ is an upper rate function for $B_t$ with respect to $\rho_o$, where $\psi$ is given by 
 \begin{equation}
  t=\int_{2}^{\psi(t)}
         \frac{r}{\log \mathrm{vol}(B(r))+\log\log r}
  \,dr. 
\end{equation}
This shows that Corollary \ref{Coro-explosion} covers \cite[Theorem 4.1]{HQ10}.

In the case when $A$ is the identity operator and $V\neq 0$, the associated infinitesimal generator of the Dirichlet form is given by the diffusion operator $L=\varDelta+\nabla V$ on $\C_0^\infty(M)$. The intrinsic metric is still the Riemannian distance. Hence,  there is some constant $C>0$ such that $\psi(Ct)$ is an upper rate function with respect to $\rho_o$ for the associated diffusion process, where $\psi$ is given by 
\begin{equation}
  t=\int_{2}^{\psi(t)}
         \frac{r}{\log \m(B(r))+\log\log r}
  \,dr. 
\end{equation}

For the general case, let 
\[
            \lambda(r):=\sup_{y\in B(r)} \<A \nabla \rho_o,\nabla \rho_o\>(y)  .
\]
By Theorem \ref{Theoem_rate}, there is a constant $C>0$ such that $\psi(Ct)$ is an upper rate function with respect to $\rho_o$, where $\psi(t)$ is given by 
\begin{equation}
  t=\int_{2}^{\psi(t)}
         \frac{r}{\lambda(r)\left(\log \m(B(r))+\log\log r\right)}
  \,dr. 
\end{equation}

In particular, if we have 
\[
	\<A\nabla \rho_o,\nabla \rho_o\>\leq c \rho_o^\gamma
\]
for some positive constants $c$ and $\gamma$, then there is a constant $C>0$ such that  $\psi(Ct)$ is an upper rate function, where $\psi$ is given by 
\begin{equation}
  t=\int_{2}^{\psi(t)}
         \frac{r^{1-\gamma}}{\log \m(B(r))+\log\log r}
  \,dr. 
\end{equation}
\end{example} 

\begin{example} 
Let $U=\{x\in\R^n\colon |x|<l\}$ be the Euclidean ball with center the origin and  radius $l>0$.  Let $\Xi,\Phi$ be continuous positive functions on $[0,\infty)$. Set 
\[
	\xi(x)=\Xi(|x|),\quad \phi(x)=\Phi(|x|),\quad x\in U
\]
and $d\m(x)=\phi^2\xi^2dx$. Consider the Markovian form 
$(\EC,\C_0^\infty(U))$ with 
\[
	\EC(f,g)=\int_U \<\nabla f,\nabla g\> \xi^{-2}
			\,d\m	
\]
on the Hilbert space $L^2(U,d\m)$. 

The intrinsic metric $d$ is given by 
$$d(x,0)=\int_0^{|x|} \Xi(s)ds,\quad x\in U.$$
Suppose that 
$$\int_0^{l} \Xi(s)ds=\infty.$$
Clearly the form $(\EC,\C_0^\infty(U))$ is closable and Assumption \ref{Top_Assumption} holds. 
Hence $(U,d)$ is complete. The volume of the ball 
$$B(0,r)=\{x\in U\colon d(x,0)\leq r\}
           =\left\{x\in U\colon \int_0^{|x|}\Xi(s)\,ds\leq r\right\}$$ 
with center  $0$ and radius $r>0$ is given by 
\[
	\m(B(0,r))=\int_{B(0,r)}\,\phi^2\xi^2dx 
	      =\mathrm{vol}(\SS^{n-1}) 
	      	\int_0^{s^{*}(r)}
		 \Phi^2(s)\Xi^2(s)s^{n-1}\,ds,
\]
where $s^{*}(r)$ is determined by 
$$\int_0^{s^{*}(r)}\Xi(s)\,ds=r, 
$$
and 
$\mathrm{vol}(\SS^{n-1}) $ is the volume of the standard sphere 
$\SS^{n-1}:=\{x\in \R^n\colon
		|x|=1
 \}$.

Suppose that 
$l=+\infty$ (i.e. $U=\R^n$) and 
$$\Phi(s)  \approx s^a,\quad \Xi(s)  \approx s^b$$
with $a,b>0$. 
Here for any two functions $f$ and $g$,  $f \approx g$ means that there exists a constants $C(f,g)>0$ such that 
\[\frac{1}{C(f,g)}g\leq f\leq C(f,g)g. \]
Then for any $x\in U$ we have 
\[
	d(x,0)=\int_0^{|x|}\Xi(s)ds\approx |x|^{1+b}. 	
\]
Hence 
\[s^*(r)\approx r^{\frac{1}{1+b}}.\]
So 
\[
	\m(B(0,r))\approx \int_0^{s^*(r)} s^{2(a+b)+n-1}ds
	\approx (s^*(r))^{2(a+b)+n}\approx r^{\frac{2(a+b)+n}{1+b}}. 
\]
Therefore, 
by Theorem \ref{Theoem_rate}, 
for any $x\in\R^n$, 
there exists some constant $C>0$ such that 
\[
   \P_x(d(X_t,0)\leq C \sqrt{t\log t} \ \mathrm{for\ all\ }t\ \mathrm{sufficiently\ large})=1.
\]
In terms of the Euclidean metric, it implies that there exists some constant $C_1>0$ such that 
\[
   \P_x(|X_t|\leq C_1 (t\log t)^{\frac1{2(1+b)}} \ \mathrm{for\ all\ }t\ \mathrm{sufficiently\ large})=1.
\]
\end{example}

\begin{example}\label{Example:Diri-Form-Rn}
Let $\EC$ be a symmetric bilinear form on $L^2(\R^n,dx)$ defined by 
\[
   \EC(f,g)=\sum_{i,j=1}^n\int_{\R^n} 
      a_{ij}(x)\frac{\partial f}{\partial x_i}
	       \frac{\partial g}{\partial x_i}	dx,\quad f,g\in \C_0^\infty(\R^n),
\]
where for each $1\leq i,j\leq n$, $a_{ij}$ is a locally integrable measurable function on $\R^n$ such that 
$(a_{ij})_{n\times n}$ is symmetric, locally uniformlly elliptic and for all $\xi\in\R^n$
\begin{equation}\label{ax_cond_form_on_Rn}
 	\sum_{i,j=1}^n  a_{ij}(x) \xi_i\xi_j 
		         \approx a(x)\|\xi\|^2
\end{equation}
for some positive  continuous function $a(x)$ on $\R^n$.

It is well known that $(\EC,\C_0^\infty(\R^n))$ is closable (see \cite[Section 3.1]{FOT94}) and its closure $(\EC,\F)$  is a strongly local,  regular, conservative and symmetric Dirichlet form on $L^2(\R^n,dx)$.

Define a distance function $d$ on $\R^{n}\times \R^n$ by 
\begin{equation*}
\begin{aligned}
	d(x,y)=\inf\left\{
		\int_0^1 a^{-1/2}(\gamma(s))|\gamma'(s)|ds\colon 
		\gamma\in \C^1([0,1];\R^n),\ \gamma(0)=x,\ \gamma(1)=y
	\right\}
\end{aligned}
\end{equation*}
for all $x,y\in \R^n$. 
Essentially $d$ is proportional to the intrinsic metric. 
For every $x\in\R^n$, 
let $\rho_x(y)=d(x,y)$ for all $y\in \R^n$. Set for all $r>0$, 
$B_{\rho_x}(r):=\{y\in \R^n\colon  \rho_x(y)\leq r \}$. 
It is clear that 
$\rho_x$ satisfies Assumption \ref{Assumption_rho}. Moreover, we have for all $r>0$, 
$$\lambda(r)=\sup_{y\in B_{\rho_x}(r)}
	\sum_{i,j=1}^n 
      a_{ij}(y)\frac{\partial \rho_x}{\partial x_i}(y)
	       \frac{\partial \rho_x}{\partial x_i}(y)\approx 1
.$$
By Corollary \ref{cora_of_Theoem_rate}, for any $x\in\R^n$ there is a constant $C>0$ such that 
$\psi(Ct)$ is an upper rate function  with respect to $\rho_x$ for the associated diffusion process $X_t$, where 
$\psi$ is given by
\begin{equation}\label{psi_a-diri}
  t=\int_{2}^{\psi(t)}
         \frac{r}{\log \mathrm{vol}(B_{\rho_x}(r))+\log\log r}
  \,dr. 
\end{equation}
Here $\mathrm{vol}(B)$ is the volume of Borel set $B\subset \R^n$ with respect to the Lebesgue measure.
\end{example}

\begin{example}\label{Example:Diri-Form-Rn-radial-a}
We proceed to consider Example \ref{Example:Diri-Form-Rn}. Suppose that the function $a(x)$
in \eqref{ax_cond_form_on_Rn} 
 is radial, i.e. there exists some strictly positive function $\tilde{a}$ on $[0,\infty)$ such that 
 \[
 	a(x)=\tilde{a}(|x|),\quad x\in\R^n. 
 \]
Then we have 
\[
	\rho_x(y)\approx\tilde{\rho}(|x-y|),\quad x,y\in\R^n, 
\]
where 
\begin{equation}\label{Equ:rho-p-int}
	{\tilde\rho}(s)= \int_0^{s}\frac{1}{\sqrt{\tilde{a}(u)} }\,du,\quad s\in[0,\infty). 
\end{equation}

So there exists some constant $C>0$ such that 
\begin{equation}\label{Equ:rate_fnt_psi-example3-4}
   \P_x(\rho_x(X_t)\leq C\psi(C t)\ \mathrm{for\ all\ }t\ \mathrm{sufficiently\ large})=1, 
\end{equation} 
where by \eqref{psi_a-diri}, $\psi$ can be represented as
\begin{equation}\label{Equ:rate-psi}
  t=\int_{2}^{\psi(t)}
        			   \frac{r}{n \log [\tilde{\rho}^{-1}(r)]+\log\log r}
  \,dr .
\end{equation}
Let 
\begin{equation}\label{equ:tilde-psi-from-psi}
	\tilde{\psi}=\tilde{\rho}^{-1}\circ \psi.
\end{equation}
We have 
\begin{equation}\label{Equ:rate_fnt_tilde_psi}
   \P_x(|X_t|\leq C_1\tilde{\psi}(C_1 t)\ \mathrm{for\ all\ }t\ \mathrm{sufficiently\ large})=1
\end{equation} 
for some constant $C_1>0$.

Let us consider  three special cases of function  $a$ and look for $\psi$ and $\tilde{\psi}$ satisfying  \eqref{Equ:rate_fnt_psi-example3-4}
and \eqref{Equ:rate_fnt_tilde_psi} for some constants $C>0$ and $C_1>0$ respectively.

\textbf{Case 1.}
Suppose that $a\equiv 1$. Then    
  $\tilde{\rho}(s)=s$. So we have
  $\psi(t)=\sqrt{t\log t}$
and $\tilde{\psi}(t)=\sqrt{t\log t}.$

\textbf{Case 2.}
Suppose that 
  \(a(x)=(1+|x|)^{\alpha}\) for some $\alpha<2$. 
Then  $\tilde{\rho}(s)=(1+s)^{1-\alpha/2}$.
So we have  
   $\psi(t)=\sqrt{t\log t}$ and 
    $\tilde{\psi}(t)=(t\log t)^{1/(2-\alpha)}$.

\textbf{Case 3.}
Suppose that \(a(x)=(1+|x|)^2[\log(1+|x|)]^\beta
  \)
  for some $\beta\leq 1$. 
Then  $\tilde{\rho}(s)=[\log(1+s)]^{1-\beta/2}$.
We have: 
  \begin{enumerate}
   \item If $\beta<1$, then we have 
 $\psi(t)=t^{1+\frac{\beta}{2-2\beta}}$ and 
    $\tilde{\psi}(t) = \exp( t^{\frac{1}{1-\beta}})$. 
  \item If $\beta=1$, then 
  we have 
 $\psi(t)= \exp(t)$ and 
    $\tilde{\psi}(t)=\exp(\exp(t)) $. 
  \end{enumerate}
\end{example}

\begin{remark}
In Example \ref{Example:Diri-Form-Rn-radial-a}, if $a(x)=(1+|x|)^\alpha$ with $\alpha>2$ or $a(x)=(1+|x|)^2\log(1+|x|)^\beta$ with $\beta>1$, then the corresponding Dirichlet form is not conservative (see \cite[Example B and Note 6.6]{Dav85}). 
\end{remark}

\begin{remark}
    To get an upper rate function with respect to the Euclidean metric for the process considered in Example \ref{Example:Diri-Form-Rn-radial-a},
usually it is convenient to apply Theorem \ref{Theoem_rate} with $\lambda(r)=\tilde{a}(r)$. That is, one can get  
$\tilde{\psi}$ satisfying \eqref{Equ:rate_fnt_tilde_psi}
by solving 
	\begin{equation}\label{Equ:rate-psi-tilde}
  t=\int_{2}^{\tilde{\psi(t)}}
        			   \frac{r}{\tilde{a}(r)(n \log (r)+\log\log r)}
  \,dr.
\end{equation}
For Case 1 and Case 2, from 
\eqref{Equ:rate-psi-tilde}
 we get the same functions $\tilde{\psi}(t)$ with those in Example \ref{Example:Diri-Form-Rn-radial-a}. 
However, for Case 3,  if $\beta>0$, we have 
$$
	\int_{2}^{+\infty}
        				   \frac{1}{r( \log (r))^{1+\beta}}
		  \,dr<\infty, 
$$
 so we cannot get $\tilde{\psi}$ from \eqref{Equ:rate-psi-tilde}. 
For $\beta= 0$ and $\beta<0$, by \eqref{Equ:rate-psi-tilde}, 
we have 
$\tilde{\psi}(t)=\exp(t)$ and $\tilde{\psi}(t)=\exp(t^{-\frac{1}{\beta}})$ respectively.
Clearly for the case  $\beta= 0$ we obtain the same function 
$\tilde{\psi}(t)$ with the one obtained in Example \ref{Example:Diri-Form-Rn-radial-a}.
But for the case $\beta<0$, we get less precise upper rate function. This is the cost we have to pay for the convenience  of using $\lambda(r)=\tilde{a}(r)$. 
\end{remark}

\begin{remark}
Comparing with Example \ref{Exa:L-rho-aI} in the next section, it turns out that the method using volume growth needs less information of the Dirichlet form, but sometimes gives less exact upper rate function.  
\end{remark}

\section{Comparison theorem for escape rates}\label{Sec:Comparison-Thm}
Let $M$ be a $n$-dimensional complete smooth connected Riemannian manifold. 
Consider a diffusion operator 
\begin{equation}\label{Equ:B-L-operator}
       L=\varDelta + Z
\end{equation}
on $M$, where $\varDelta$ is the Laplace-Beltrami operator on $M$ and $Z$ is a $\C^1$ vector field on $M$. 

Let $o\in M$ be a fixed reference point and set $$\rho_o(x)=d(x,o)$$
 for every $x\in M$. Here $d(\cdot,\cdot)$ is the Riemannian distance function on $M$.

Let $\cut(o)$ denote the cut locus of $o$. 
Suppose that there exists some measurable function $\theta$ on $[0,+\infty)$ such that 
\begin{equation}\label{Equ:L_rho<theta}
L\rho_o\leq \theta(\rho_o)
\end{equation}
on $M\setminus \cut(o)$. 

Note that by a comparison theorem of Bakry and Qian \cite[Theorem 4.2]{BQ05}, 
Inequality \eqref{Equ:L_rho<theta} follows from the curvature-dimension condition (see \cite{Bak94}). 

Let $(X_t,\zeta,\P_x)_{x\in M}$ be the diffusion process on $M$ associated with $L$. Here $\zeta$ is the life time of $X_t$. 
Let  $B_R$ denote the closed geodesic ball  with center $o$ and radius $R>0$. Set $ B_R^o=B_R\setminus \partial B_R$. 
Let $\tau_R$ denote the first exit time of $X_t$ from $B_R$. That is, 
$$
	\tau_R:=\inf\{t\geq 0\colon X_t\notin B_R \}.
$$

Let $(x_t,\P^0_r)$ be the solution to the following stochastic differential equation 
\begin{equation}\label{Equ:radial-process}
	dx_t= \theta(x_t)dt+\sqrt{2}dw_t,\quad x_0=r\geq 0
\end{equation}
on $[0,+\infty)$. Here $w_t$ is a standard Brownian motion, $\theta$ is the function on $[0,+\infty)$ satisfying \eqref{Equ:L_rho<theta}.  
The infinitesimal generator of $x_t$ is given by 
\[
    L_0=\frac{\partial^2}{\partial r^2}+\theta(r)\frac{\partial}{\partial r}.
\]
For every $0<R<\delta_K$, let $\tau_R^0$ be the first exit time of $x_t$ from $[0,R]$. That is,   
$$
	\tau_R^0:=\inf\{t\geq 0\colon x_t\notin [0,R] \}.
$$

We have the following comparison theorem for the upper rate functions of $X_t$ and $x_t$. To some extend, it is a generalization of \cite[Theorem 2.1]{Ich88}, \cite[Proof of Inequality (4.6)]{Stu92} and \cite[Proof of Inequality (2.2)]{GW01}. 

\begin{theorem}\label{Thm:SemigroupComp}
Suppose that there is some measurable function $\theta$ on $[0,+\infty)$ such that  \eqref{Equ:L_rho<theta}
 holds on $M\setminus \cut(o)$. 
Let $R>0$ and $x\in M$ with $r:=\rho_o(x)<R$. 
Then for every $t>0$, $0<\delta<R$, 
\begin{equation}\label{equ:Thm:SemigroupComp}
  \P^0_{r}(x_t<\delta,\, t<\tau_R^0) \leq \P_{x}(\rho_o(X_t)<\delta,\, t<\tau_R)  .
\end{equation}
Therefore, an upper rate function for  $x_t$ is also an upper rate function for $X_t$.     
\end{theorem}

\begin{proof}
Let $\phi$ be a monotone, non-increasing $\C^\infty$-function on $[0,R)$ with compact support in $[0,\delta)$. Set for all $t\geq 0$, $x\in B_R$ and $r\in [0,R]$, 
 \begin{equation}
 u(t,x)=\E_x[\phi(\rho_o(X_t)),\, t<\tau_R]
\end{equation} 
 and 
\begin{equation}\label{equ:u_0_def} 
u_0(t,r)=\E_r[\phi(x_t),\, t<\tau_R^0].
\end{equation}
 It is clear that $u(t,x)\in \C^\infty((0,\infty)\times B_R)$ and $u_0(t,r)\in \C^\infty((0,\infty)\times [0,R])$. Moreover, $u(t,x)$, $u_0(t,r)$ satisfy the following two equations
 \[
 \left\{
 \begin{aligned}
    & \frac{\partial u}{\partial t}-Lu=0 \quad  \mathrm{in}\ (0,\infty)\times  B_R^o,\\
    & u(0,x)=\phi(d(x,o)), \\
    & u(t,x)=0 \quad \mathrm{on}\ \partial B_R^o 
 \end{aligned}
 \right.
 \]
 and
  \[
  \left\{
 \begin{aligned}
    & \frac{\partial u_0}{\partial t}-L_0 u_0=0 \quad  \mathrm{in}\ (0,\infty)\times (0,R),\\
    & u_0(0,r)=\phi(r),\\
    & u_0(t,R)=0
 \end{aligned}
 \right.
 \]
respectively. 

It is clear that $x_t$ is monotone non-decreasing relative to the initial value $r$  (cf. \cite[Lemma 2.1]{Ich88}). 
Since $\phi$ is monotone non-increasing, by \eqref{equ:u_0_def} we get that  the function $u_0(t,r)$ is monotone non-increasing in $r<R$. So 
$$\frac{\partial}{\partial r} u_0(t,r) \leq 0. $$
Let $v_0(t,x)=u_0(t,\rho_o(x))$.
By \eqref{Equ:L_rho<theta},
 for all $x\notin \cut(o)$ with  $\rho_o(x)=r<R$,  we have
  \[
   \begin{aligned}
 \frac{\partial}{\partial t}v_0(t,x)= \frac{\partial}{\partial t}u_0(t,r)\left|_{r=\rho_o(x)}\right. 
 &=
    L_0 u_0(t,r)  \left|_{r=\rho_o(x)}\right. 
\\
      & =\left(\frac{\partial^2}{\partial r^2} +\theta(r)
             \frac{\partial }{\partial r}
             \right)  u_0(t,r) \left|_{r=\rho_o(x)}\right. 
 \\
       &\leq \left( \frac{\partial^2}{\partial r^2} 
                     + (L r) \frac{\partial }{\partial r}  
             \right)u_0(t,r)  \left|_{r=\rho_o(x)}\right. 
       =L v_0(t,x)     . 
   \end{aligned}
 \] 
 Consequently we have
   \begin{equation}\label{equ1:pf1:Thm:SemigroupComp}
      \left(\frac{\partial}{\partial t}-L\right) v_0(t,x)\leq 0 \quad  \mathrm{in}\ (0,\infty)\times (B_R^o\setminus \cut(o)).
   \end{equation}
 Using similar arguments in the appendix of \cite{Yau76}, we have
    \begin{equation*}
      \left(\frac{\partial}{\partial t}-L\right) v_0(t,x) \leq 0 \quad  \mathrm{in}\ (0,\infty)\times B_R^o
   \end{equation*}
  in the distributional sense. 
  Let $U(t,x)=v_0(t,x)-u(t,x)$. Then
    \[
      \left(\frac{\partial}{\partial t}-L\right) U(t,x)\leq 0 \quad  \mathrm{in}\ (0,\infty)\times B_R^o
   \]
  in the distributional sense. Note that for all $x\in B_R^o$, $U(0,x)=0$,  and for all $t\geq 0$, $x\in \partial B_R^o$, $U(t,x)=0$, by the parabolic maximum principle, we have 
  \[U(t,x)\leq 0\]
  for every $(t,x)\in [0,\infty)\times B_R$. 
 That is,  we get 
 \begin{equation}\label{Com:Ex-Er}
 \E_r[\phi(x_t),\, t<\tau_R^0]\leq 
 	\E_x[\phi(\rho_o(X_t)),\, t<\tau_R]
 \end{equation}
 for all $(t,x)\in [0,\infty)\times B_R$ with $r=\rho_o(x)$. 
By letting $\phi\uparrow \Eins_{[0,\delta)}$ on both sides of  \eqref{Com:Ex-Er}, we obtain \eqref{equ:Thm:SemigroupComp}. 
   
If there exists an upper rate function for $x_t$, then 
\(\P_r(\zeta_0=+\infty)=1.\)
Letting $R\to +\infty$ on both sides of \eqref{equ:Thm:SemigroupComp}, we obtain 
\begin{equation}\label{equ:Thm:SemigroupComp-R-infty}
\P^0_{r}(x_t<\delta) \leq \P_{x}(\rho_o(X_t)<\delta,\, t<\tau_{\infty}).  
\end{equation}
So if $R(t)$ is an upper rate function for $x_t$, i.e. 
\[
    \P^0_{r}(x_t\leq R(t)\ \mathrm{for\ all\ sufficiently\ large}\ t)=1,
\]
then we also have 
\[
	\P_{x}(\rho_o(X_t)\leq R(t)\ \mathrm{for\ all\ sufficiently\ large}\ t)=1.
\]
This proves that $X_t$ inherits the upper rate function for $x_t$. 
\end{proof}

\begin{remark}
A probabilistic proof of Theorem \ref{Thm:SemigroupComp} is available by modifying the arguments in the proof of \cite[Lemma 2.1]{GW01}.
\end{remark}

By Theorem \ref{Thm:SemigroupComp} and Corollary \ref{Cora4:1-dim-SDE} shown in the next section we have the following result.

\begin{cora}\label{Coro-L-rho-power}
Suppose that for some $-1\leq \alpha\leq 1$ there exists some constant $K_\alpha>0$ such that 
$$
L\rho_o\leq K_\alpha \rho_o^\alpha
$$
holds on $M \setminus \cut(o)$. Then for some constant $C_\alpha>0$, $g_\alpha(C_\alpha t)$ is an upper rate function for the $L$-diffusion process $X_t$,  where $g_\alpha$ is given by 
\begin{equation}\label{Equ:psi_alpha}
 	g_\alpha(t)=\left\{
	\begin{aligned}
		& \sqrt{t\log\log t}, &\quad & \alpha=-1, \\
		& t^{\frac{1}{1-\alpha}}, & \quad  & -1<\alpha <1, \\
		& \e^t,  &\quad & \alpha=1.
	\end{aligned}	
	\right.
 \end{equation} 
\end{cora}

\begin{example}
Let $n$ be an integer and $\R^n$ be endowed with the following metric 
 \[
   ds^2=dr^2 +\xi^2(r)d\theta^2,
 \]
 where $(r,\theta)$ is the polar coordinates in $\R^n=\R_+\times \SS^{n-1}$, $\xi(r)$ is a positive smooth function on $\R_+$ satisfying $\xi(0)=0$, $\xi'(0)=1$, and $d\theta^2$ is the standard Riemannian metric on $\SS^{n-1}$. We call $M_\xi:=(\R^n,\xi)$ a model manifold.
 
 The Laplace-Beltrami operator $\varDelta$ on $M_\xi$ 
can be written as follows
\[
  \varDelta=\frac{\partial^2}{\partial r^2} +m(r)\frac{\partial }{\partial r}
  +\frac{1}{\xi^2(r)}\varDelta_{\theta},
\]
where $\varDelta_\theta$ is the standard Laplace operator on $\SS^{n-1}$ and $m(r)$ is the mean curvature function of $M_\xi$ given by
\[
      m(r)=(n-1)\frac{\xi'(r)}{\xi(r)}.
\]
It is clear that we have $\varDelta r = m(r) $. 

In particular, let us take $\xi(r)=\sinh\sqrt{K}r$ for some constant $K>0$. 
Then 
$M_\xi$ is the hyperbolic space $\mathbb{H}^n$, i.e. 
the complete simply connected $n$-dimensional manifold with
constant sectional curvature $-K$. 
Clearly we have 
\[
	\varDelta r = (n-1)\sqrt{K}\coth\sqrt{K}r,\quad r>0.
\]
Note that $ (n-1)\sqrt{K} \coth\sqrt{K}r\to (n-1)\sqrt{K}$ as $r\to +\infty$. 
Hence, by Proposition \ref{Prop3:1-dim-SDE}, for every $\varepsilon>0$,  $(1+\varepsilon)(n-1)\sqrt{K}t$ is an upper rate function for the Brownian motion on $M_\xi$.
\end{example}

\begin{example}\label{Exa:L-rho-aI}
Consider the Dirichlet form $(\EC,\F)$ on $L^2(\R^n,dx)$ given by 
\[
   \EC(f,g)=\sum_{i,j=1}^n\int_{\R^n} 
      a_{ij}(x)\frac{\partial f}{\partial x_i}
	       \frac{\partial g}{\partial x_i}	dx,\quad f,g\in \C_0^\infty(\R^n),
\]
where $a_{ij}$ are continuously differentiable functions on $\R^n$ such that $(a_{ij})$ is positive definite. 
The infinitesimal generator associated with $(\EC,\F)$ is given by 
\begin{equation}\label{Equ:inf-generator1}
	L
	=\sum_{i,j=1}^n \frac{\partial}{\partial x_i}
		\left( a_{ij}(x) \frac{\partial}{\partial x_i}\right)
	=\sum_{i,j=1}^n a_{ij}(x)\frac{\partial^2}{\partial x_i\partial x_j}
			+\sum_{i}^n b_i(x)\frac{\partial}{\partial x_i}
\end{equation}
on $\C^2(\R^n)$, where \[
		b_i(x)=\sum_{j=1}^n \frac{\partial a_{ij}}{\partial x_j}(x).
\]
 We introduce a new Riemannian metric $g=(g_{ij}(x))=(a_{ij})^{-1}$ on $\R^n$. Then 
$$L=\varDelta_g + \nabla_g \left(\frac12\log \det a\right).$$
Here $\varDelta_g$ and  $\nabla_g$ are the Beltrami-Laplace operator and gradient operator on $(\R^n,g)$ respectively. The Riemannian distance function $\rho$ on $(\R^n,g)$ is the same with the intrinsic metric of the Dirichlet form $(\EC,\F)$. 

Suppose that the Riemannian distance function is given by  $\rho_x(y)=\tilde{\rho}(|x-y|)$, $x,y\in\R^n$, for some positive function $\tilde{\rho}\in \C^2([0,\infty))$.  
Then 
\[
	L{\rho}_0(x)=A(x)\tilde{\rho}''(|x|)+\frac{\tilde{\rho}'(|x|)}{|x|}
		\left[ B(x)-A(x)+C(x) \right],
\]
where
\[\begin{aligned}
	A(x)&=\frac{1}{|x|^2}\sum_{i,j=1}^n a_{ij}(x)x_ix_j,\\
	B(x)&=\sum_{i=1}^n a_{ii}(x),\\
	C(x)&=\sum_{i=1} x_ib_i(x).	
\end{aligned}
\]

In particular, let us consider a toy model to illustrate the application. Suppose that $(a_{ij})=\tilde{a}(|x|)I$ for some strictly positive and continuously differentiable function $\tilde{a}$ on $[0,\infty)$. As in Example \ref{Example:Diri-Form-Rn-radial-a},  for the Riemannian distance function we have 
$$\rho_x(y)\approx\tilde{\rho}(|x-y|),\quad x,y\in\R^n,$$ where 
$${\tilde\rho}(s)= \int_0^{s}\frac{1}{\sqrt{\tilde{a}(s)} }\,ds.$$ 

Let $r=|x|$. Then 
\[
	L{\rho}_0(x)=-\frac{ \tilde{a}'(r) }{2 \sqrt{\tilde{a}(r)}}
			+ \frac{(n-1)\sqrt{\tilde{a}(r)}}{r}.					
\]

Now we apply Corollary \ref{Coro-L-rho-power} to the three cases investigated  in Example \ref{Example:Diri-Form-Rn-radial-a}. That is, we look for functions  $\psi$ and $\tilde{\psi}$ satisfying  \eqref{Equ:rate_fnt_psi-example3-4}
and \eqref{Equ:rate_fnt_tilde_psi} for some constants $C>0$ and $C_1>0$ respectively. 

\textbf{Case 1.}
Suppose that $a\equiv 1$. Then    
  $\rho_x(y)=|x-y|$ and $$L\rho_x=\frac{n-1}{\rho_x}.$$
So we have  
   $\psi(t)=\sqrt{t\log\log t}$
and
$\tilde{\psi}(t)=\sqrt{t\log\log t}.$

\textbf{Case 2.}
Suppose that 
  \(a(x)=(1+|x|)^{\alpha}\) for some $\alpha<2$. 
Then 
  $\rho_x(y)\approx|x-y|^{1-\alpha/2}$. 
So we have 
$$L\rho_x\leq  \frac{C'}{\rho_x} $$
 for some constant $C'>0$. 
We get 
   $\psi(t)=\sqrt{t\log\log t}$ and  
    $\tilde{\psi}(t)=(t\log \log t)^{1/(2-\alpha)}$.

\textbf{Case 3.}
Suppose that \(a(x)=(1+|x|)^2[\log(1+|x|)]^\beta
  \)
  for some $\beta\leq 1$. 
Then 
  $\rho_x(y)\approx [\log(1+|x-y|)]^{1-\beta/2}$.
Hence for some constant $C'>0$, 
 $$ L\rho_x\leq C' \rho_x^{\frac{\beta}{2-\beta}}.$$  
We get the same functions $\psi$ and $\tilde{\psi}$ as in Case 3 of Example \ref{Example:Diri-Form-Rn-radial-a}.
\end{example}

\section{Escape rate of one dimensional It\^o process}\label{Sec:escape-one-dim} 
Following the arguments in \cite[\S 17]{GS72}, we include here a short study of 
the upper rate function  (with respect to the Eculidean metric) for a general one dimensional It\^o diffusion process. This is useful for the application of Theorem \ref{Thm:SemigroupComp}.  

Consider the following one dimensional stochastic differential equation
\begin{equation}\label{Equ:1-dim-SDE}
	dz_t=b(z_t)dt + \sigma(z_t)dw_t,\quad t\geq 0,  
\end{equation}
where $b$ and $\sigma>0$ are measurable functions on $[0,+\infty)$, $w_t$ is the Brownian motion on $[0, +\infty)$. 

We start from the following simple result. 

\begin{prop}\label{Prop1:1-dim-SDE}
	Let $z_t$ satisfy \eqref{Equ:1-dim-SDE}. Suppose that the following conditions hold: 
	\begin{enumerate}
		\item For large enough $x>0$, $b(x)$ is bounded above by some constant $b_0>0$, i.e. $$b(x)\leq b_0\quad \textrm{as}\ x\to +\infty.$$
		\item There exist some constants $\alpha<1$ and $C_\sigma>0$ such that for all $x> 0$,
		\begin{equation}\label{Equ0:Prop1:1-dim-SDE} 
			   \sigma^2(x)\leq C_\sigma(1+x^\alpha) . 
		 \end{equation}  
	\end{enumerate}
Then for every $\varepsilon>0$, $(b_0+\varepsilon)t$ is an upper rate function for $z_t$. 
\end{prop}

\begin{proof}
As $z_t$ satisfies \eqref{Equ:1-dim-SDE}, we have 
\begin{equation}\label{Equ1:Prop1:1-dim-SDE}
 	z_t=z_0+ \int_0^t b(z_s)ds +\int_0^t \sigma(z_s)dw_s. 
 \end{equation}
 By \eqref{Equ0:Prop1:1-dim-SDE} we have (cf. \cite[\S 17, Lemma 1]{GS72})
\begin{equation}\label{Equ3:Prop1:1-dim-SDE}
	\P_z\left(\lim_{t\to +\infty} \frac{1}{t} 
		\int_0^t \sigma(z_s)dw_s =0
	\right)=1.
\end{equation}
We only need to consider the path of $z_t$ which wanders out to infinity. Because $b(x)\leq b_0$ as $x\to +\infty$, 
we have 
\begin{equation}\label{Equ2:Prop1:1-dim-SDE}
	\int_0^t b(z_s)ds\leq b_0 t
\end{equation}
for sufficiently large $t>0$.  
Thus the proof is complete by combining 
\eqref{Equ1:Prop1:1-dim-SDE}, \eqref{Equ3:Prop1:1-dim-SDE} and 
\eqref{Equ2:Prop1:1-dim-SDE}. 
\end{proof}

With the help of the previous proposition and It\^o's formula, we have the following result. 

\begin{prop}\label{Prop2:1-dim-SDE}
Let $z_t$ satisfy \eqref{Equ:1-dim-SDE}.  Let $f(x)$ be  an increasing, twice continuously differentiable function and $g$ the inverse function of $f$. 
Suppose that the following conditions hold: 
\begin{enumerate}
	\item There exists a constant $b_0>0$ such that for large enough $x>0$, 
		\begin{equation}\label{Equ1:Prop2:1-dim-SDE}
			b(g(x))f'(g(x))+\frac12\sigma^2(g(x))f''(g(x))\leq b_0.
		\end{equation}
	\item There exist some constants $C>0$, $\alpha<1$ such that for all $x>0$,
		\begin{equation}\label{Equ2:Prop2:1-dim-SDE}
			\sigma(g(x))f'(g(x))\leq C(1+x^\alpha). 
		\end{equation}	
\end{enumerate}
Then for $\varepsilon>0$, $g((b_0+\varepsilon)t)$ is an upper rate function for $z_t$.
\end{prop}

\begin{proof}
Let $\tilde{z}_t=f(z_t)$. By It\^o's formula we have
\begin{equation}\label{Equ4:Prop2:1-dim-SDE} 
		d\tilde{z}_t=\hat{b}(\tilde{z}_t)dt + \hat{\sigma}(\tilde{z}_t)dw_t,  
\end{equation}
where 
\[
	\begin{aligned}
		\hat{b}(x)&=b(g(x))f'(g(x))+\frac12\sigma^2(g(x))f''(g(x)),\\
		\hat{\sigma}(x)&=\sigma(g(x))f'(g(x)).
	\end{aligned}
\]
By Proposition \ref{Prop1:1-dim-SDE} as well as Conditions \eqref{Equ1:Prop2:1-dim-SDE} and \eqref{Equ2:Prop2:1-dim-SDE},  for every  $\varepsilon>0$, 
$(b_0+\varepsilon)t$ is an upper rate function for $\tilde{z}_t$. So 
$g((b_0+\varepsilon)t)$ is an upper rate function for $z_t$ since $z_t=g(\tilde{z}_t)$. Thus the proof is complete. 
\end{proof}

\begin{prop}\label{Prop3:1-dim-SDE}
Let $z_t$ satisfy \eqref{Equ:1-dim-SDE}. 
Let $\tilde{b}$ be a positive function on $[0,+\infty)$ such that 
$b(x)\leq \tilde{b}(x)$ holds for large enough $x>0$ and
 \(
	\int^{+\infty} \frac{1}{\tilde{b}(s)}\, ds=+\infty.  
\)
Let $g$ be a function on $[0,\infty)$ definded  by
\(
	t=\int_0^{g(t)} \frac{1}{\tilde{b}(s)}\, ds. 
\)
Suppose that the following conditions hold:
\begin{enumerate}
\item There exist some constants $C_1>0$ and $\alpha<1$ such that for all $x>0$, 
	\[
		\frac{\sigma(g(x))}{\tilde{b}(g(x))}\leq C_1(1+x^\alpha).
	\]
\item There exists some constant $C_2>0$ such that for all $x>0$, 
	\[
		-\left( \frac{\sigma(x)}{\tilde{b}(x)} \right)^2 \tilde{b}'(x)\leq C_2. 
	\]
\end{enumerate}
Then for every $\varepsilon>0$, $g ((C+\varepsilon)t)$ is an upper rate function for $z_t$, 
where $C=1+C_2/2$.
\end{prop}

\begin{proof}
	Let 
		$$ f(x)=\int_0^x	 \frac{1}{\tilde{b}(u)} du$$
	for sufficiently large $x>0$. Then we extend $f$ in such a way that it is increasing on the rest of the half line $[0,\infty)$, and that $f'$, $f''$ exist. 	It is clear that for sufficiently large $x>0$, 
\[
	\begin{aligned}
		\sigma(g(x))f'(g(x)) = \frac{\sigma(g(x))}{\tilde{b}(g(x))}\leq C_1(1+x^\alpha),
\end{aligned}
\]
and
\[			
\begin{aligned}
		   &  b(g(x))f'(g(x))+\frac12\sigma^2(g(x))f''(g(x)) \\
		 = & \frac{b(g(x))}{\tilde{b}(g(x))}
		-\frac12 \left( \frac{\sigma(g(x))}{\tilde{b}(g(x))} \right)^2 \tilde{b}'(g(x))
		\leq 1+\frac12 C_2.
	\end{aligned}
\]	
So Conditions \eqref{Equ1:Prop2:1-dim-SDE} and \eqref{Equ2:Prop2:1-dim-SDE} are satisfied. 
Hence the proof is finished by applying Proposition \ref{Prop2:1-dim-SDE}.  
\end{proof}

%

As an application of Proposition \ref{Prop3:1-dim-SDE}, let us give an upper rate function for $x_t$ satisfying \eqref{Equ:radial-process}.
\begin{cora}\label{Cora4:1-dim-SDE}
Let $x_t$ satisfy \eqref{Equ:radial-process}. Suppose that 
for some $-1\leq \alpha\leq 1$, there exists some constant $\theta_\alpha>0$ such that  
$$\theta(x)\leq \theta_\alpha x^\alpha $$
holds for large enough $x>0$.  
Then for some constant $C_\alpha>0$, $g_\alpha(C_\alpha t)$ is an upper rate function for $x_t$, 
where $g_\alpha$ is given by
\eqref{Equ:psi_alpha}. 
\end{cora}
\begin{proof}
	The case when $\alpha=-1$ follows from \cite[Chapter 2, Theorem 5.4]{Mao08}.
For the case $-1< \alpha \leq 1$, we only need to apply Proposition \ref{Prop3:1-dim-SDE} with $\sigma=\sqrt{2}$ and $\tilde{b}(x)=\theta_\alpha x^\alpha$ for large enough $x>0$. The function $g_\alpha(t)$ is obtained by solving the following equation  
\[
	t=\int_0^{g_\alpha(t)} \frac{1}{\theta_\alpha x^\alpha} dx.
\]	
\end{proof}

\def\cprime{$'$}

%

\end{document}